\documentclass[11pt]{article}

\usepackage{amsmath,amsfonts,amsthm,amssymb,graphicx,tikz}
\usepackage[colorlinks=true, linkcolor=blue, citecolor=blue]{hyperref}

\usepackage[all,2cell,ps]{xy}

\theoremstyle{plain}
\newtheorem{thm}{Theorem}[section]

\newtheorem{lem}[thm]{Lemma}
\newtheorem{prop}[thm]{Proposition}

\newtheorem{cor}[thm]{Corollary}

\newtheorem{proj}{Project}
\newtheorem{qtn}[proj]{Question}

\theoremstyle{definition}

\newtheorem{rem}[thm]{Remark}

\theoremstyle{remark}

\newcommand{\bbB}{\mathbb{B}}

\newcommand{\bbF}{\mathbb{F}}

\newcommand{\bbP}{\mathbb{P}}
\newcommand{\bbQ}{\mathbb{Q}}
\newcommand{\bbR}{\mathbb{R}}

\newcommand{\bbZ}{\mathbb{Z}}

\newcommand{\calC}{\mathcal{C}}

\newcommand{\calG}{\mathcal{G}}

\newcommand{\calL}{\mathcal{L}}

\newcommand{\calO}{\mathcal{O}}

\newcommand{\calT}{\mathcal{T}}

\newcommand{\fraka}{\mathfrak{a}}

\newcommand{\frakp}{\mathfrak{p}}

\newcommand{\al}{\alpha}
\newcommand{\gam}{\gamma}
\newcommand{\Gam}{\Gamma}

\newcommand{\de}{\delta}
\newcommand{\Del}{\Delta}
\newcommand{\ep}{\epsilon}

\newcommand{\lam}{\lambda}

\newcommand{\Sig}{\Sigma}

\DeclareMathOperator{\Orth}{O}
\DeclareMathOperator{\M}{M}
\DeclareMathOperator{\SL}{SL}
\DeclareMathOperator{\PSL}{PSL}
\DeclareMathOperator{\GL}{GL}
\DeclareMathOperator{\PGL}{PGL}
\DeclareMathOperator{\PO}{PO}
\DeclareMathOperator{\SO}{SO}

\DeclareMathOperator{\PU}{PU}
\DeclareMathOperator{\SU}{SU}

\DeclareMathOperator{\Id}{Id}

\DeclareMathOperator{\Gal}{Gal}

\newcommand{\conj}{\overline}

\newenvironment{pf}{\begin{proof}}{\end{proof}}

\newenvironment{enum}{\begin{enumerate}}{\end{enumerate}}

\makeatletter
\let\@@pmod\pmod
\DeclareRobustCommand{\pmod}{\@ifstar\@pmods\@@pmod}
\def\@pmods#1{\mkern4mu({\operator@font mod}\mkern 6mu#1)}
\makeatother

\usetikzlibrary{arrows}

\title{Congruence RFRS towers\\ \normalsize{With an appendix by Mehmet Haluk \c{S}eng\"un}}
\author{
Ian Agol\\ \small{University of California Berkeley}
\and
Matthew Stover\\ \small{Temple University}
}
\date{\today}

\begin{document}

\maketitle

\begin{abstract}
We describe a criterion for a real or complex hyperbolic lattice to admit a RFRS tower that consists entirely of congruence subgroups. We use this to show that certain Bianchi groups $\PSL(\calO_d)$ are virtually fibered on congruence subgroups, and also exhibit the first examples of RFRS K\"ahler groups that are not a subgroup of a product of surface groups and abelian groups.
\end{abstract}

\section{Introduction}\label{sec:Intro}

Let $\Gam$ be a finitely generated group. The first author introduced the notion of $\Gam$ being \emph{virtually RFRS} to prove that certain hyperbolic $3$-manifolds are virtually fibered \cite{AgolFiber}, and eventually this was used to prove that all finite-volume hyperbolic $3$-manifolds virtually fiber \cite{AgolHaken, Wise, GrovesManning}. Finding such a cover effectively remains an open problem.

In this paper, we study finding RFRS towers arising from congruence covers of arithmetic manifolds. For example, we will prove:

\begin{thm}\label{thm:Bianchi}
The Bianchi groups $\PSL_2(\calO_d)$ with $d \not\equiv -1 \pmod 8$ and $d$ square-free contain a RFRS tower consisting entirely of congruence subgroups. In particular, these Bianchi orbifolds virtually fiber on a congruence cover.
\end{thm} 

We achieve this using the fact that these Bianchi groups virtually embed in the group $\Orth(4,1; \bbZ)$. We then apply a very general idea to the congruence subgroup of level $4$ in $\Orth(4,1; \bbZ)$ to show that it is virtually RFRS with tower $\{\Gam_j\}$ for which each $\Gam_j$ contains the congruence subgroup of level $2^{n_j}$ for some $n_j$. This example also allows us to find infinitely many commensurability classes of cocompact arithmetic Kleinian groups that virtually fiber on a congruence cover; see \cite[Lem.\ 4.6]{ALR} for examples. We will also show:

\begin{thm}\label{thm:Kahler}
There is a torsion-free cocompact lattice in $\PU(2,1)$ that is RFRS. Therefore, there is a RFRS K\"ahler group that is not isomorphic to a subgroup of the direct product of surface groups and abelian groups.
\end{thm}

This addresses a question raised in recent work of Friedl and Vidussi \cite[Qu.\ 3.2]{FV}. Our example is a congruence subgroup of a particular Deligne--Mostow lattice \cite{DeligneMostow}. Note that nonuniform lattices in $\PU(n,1)$ cannot be RFRS, since their cusp groups are two-step nilpotent groups, which themselves are not RFRS. In particular, the methods of this paper cannot apply to nonuniform complex hyperbolic lattices. See Remark \ref{rem:NotNonuniform}.

\medskip

We briefly describe the method of constructing these towers. Suppose that $k$ is a number field and $\calG$ is a $k$-algebraic group such that $\calG(k) \otimes_\bbQ \bbR$ modulo compact factors is isomorphic to $\SO(n,1)$ or $\SU(n,1)$. Let $\calO_k$ be the ring of integers of $k$ and $\frakp$ a prime ideal of $\calO_k$ with residue characteristic $p$. Suppose that $\Gam(\frakp)$ is the congruence subgroup of level $\frakp$ in the arithmetic lattice $\calG(\calO_k)$, and that $\Gam < \Gam(\frakp)$ is a finite index subgroup such that $H^1(\Gam; \bbZ)$ has no $p$-torsion.

Using the fact that $\calG(k)$ is closely related to the commensurator of $\Gam$, we find a sequence $\{g_n\}$ in $\calG(k)$ such that
\[
\bigcap_{n = 0}^\infty g_n \Gam g_n^{-1}
\]
is a RFRS tower for $\Gam$. The key is to find an initial $g_1 \in \calG(k)$ so that
\[
\Gam / (\Gam \cap g_1 \Gam g_1^{-1})
\]
is an elementary abelian $p$-group. One then inductively defines each $g_n$ in a manner most succinctly described using the $\frakp$-adic Bruhat--Tits building for $\calG(k_\frakp)$, where $k_\frakp$ is the local field associated with $\frakp$. Note that the above implicitly assumes that $H_1(\Gam; \bbZ)$ is infinite, hence our results can only apply for lattices in $\SO(n,1)$ and $\SU(n,1)$.

\medskip

We close by briefly recalling the connection between RFRS and various notions of fibering. When $\{\Gam_j\}$ is a RFRS tower with $\Gam_j = \pi_1(M_j)$ the fundamental group of an irreducible $3$-manifold, the first author proved that there is some $j_0$ so that $M_j$ fibers over $S^1$ for all $ j \ge j_0$ \cite[Thm.\ 5.1]{AgolFiber} (the statement of the theorem does not explicitly say that the fibered manifold comes from the RFRS tower, but it is implicit in the proof). It is a famous theorem of Stallings that this is equivalent to $\Gam_j$ admitting a homomorphism onto $\bbZ$ with finitely generated kernel for all $j \ge j_0$. Such a homomorphism is called an \emph{algebraic fibration}, and recent work shows that being virtually RFRS is closely related to being virtually algebraically fibered. For example, Friedl and Vidussi \cite[Thm.\ E]{FV} showed that virtually RFRS K\"ahler groups are either virtually algebraically fibered or virtually surface groups, and Kielak proved that an infinite finitely generated virtually RFRS group is virtually algebraically fibered if and only if its first $l^{(2)}$ betti number is zero \cite[Thm.\ 5.3]{Kielak}. Both proofs provide a group in the RFRS tower that algebraically fibers. See \cite{StoverBetti} and \cite{JNW} for more about algebraic fibrations of lattices in Lie groups and Coxeter groups. We rephrase our results in this language in Corollaries \ref{cor:SOfiber} and \ref{cor:CHfiber}, and see \S\ref{sec:Conclusion} for further discussion.

\medskip

This paper is organized as follows. In \S\ref{sec:Towers} we describe some basic preliminary results on RFRS towers and congruence towers. In \S \ref{sec:Exs} we give three examples that describe our general method for producing congruence towers that are RFRS. These examples suffice to prove the theorems stated above. Finally, in \S\ref{sec:Conclusion} we make closing comments and raise some questions.

\subsection{Acknowledgments} 

Agol was partially supported by the Simons Foundation and
the Simonyi professorship at the Institute of Advanced Study. Stover was supported by Grant Number 523197 from the Simons Foundation/SFARI and Grant Number DMS-1906088 from the National Science Foundation.
We thank Alan Reid and Stephen Tschantz for helpful conversations.

\section{Preliminaries on towers}\label{sec:Towers}

In this section, we discuss two types of towers of finite index subgroups of a group: RFRS towers, and $\frakp$-congruence towers.

\subsection{RFRS towers}\label{ssec:RFRS}

Let $\Gam$ be a finitely generated group with commutator subgroup $\Gam^{(1)} = [\Gam, \Gam]$ and abelianization
\[
\Gam^{ab} = \Gam / \Gam^{(1)} \cong H_1(\Gam; \bbZ).
\]
We then define the \emph{rational abelianization} $\Gam^{rab}$ to be the image of $\Gam^{ab}$ in
\[
\Gam^{ab} \otimes_\bbZ \bbQ \cong H_1(\Gam; \bbZ) \otimes_\bbZ \bbQ \cong H_1(\Gam; \bbQ)
\]
under the natural homomorphism and the \emph{rational commutator subgroup}
\[
\Gam_r^{(1)} = \ker\left(\Gam \to \Gam^{rab}\right).
\]
Clearly $\Gam^{(1)} \le \Gam_r^{(1)}$ is finite index and $\Gam^{rab} \cong H_1(\Gam; \bbZ) / \mathrm{Torsion}$.

Given a group $\Gam$, let $\{\Gam_j\}$ be a cofinal tower of finite index subgroups of $\Gam$ with $\Gam_0 = \Gam$. In other words,
\begin{enum}

\item $\bigcap \Gam_j = \{1\}$;

\item $\Gam_j$ is a finite index subgroup of $\Gam$;

\item $\Gam_{j + 1} \le \Gam_j$ for all $j$.

\end{enum}
We say that $\{\Gam_j\}$ is a \emph{RFRS tower} if, in addition,
\begin{equation}
(\Gam_j)_r^{(1)} \le \Gam_{j + 1}\ \textrm{for all}\ j \ge 0. \tag{$\star$} \label{eq:RFRS}
\end{equation}

\begin{rem}
The original definition of RFRS \cite[Def.\ 2.1]{AgolFiber} also required that $\Gam_j$ be normal in $\Gam$. However, it is also pointed out in \cite{AgolFiber} that if there is a RFRS tower, then there is also a normal RFRS tower by passing to core subgroups (i.e., the largest normal refinement).
\end{rem}

We say that $\Gam$ is \emph{RFRS} if it admits such a tower and that it is \emph{virtually RFRS} if it contains a finite index subgroup that is RFRS. We note that RFRS is short for ``residually finite $\bbQ$-solvable'', and refer to \cite{AgolFiber} for further details and examples.

We briefly recall that if $G$ is a group and $\Gam \le G$ a subgroup, the \emph{commensurator} of $\Gam$ in $G$ is the group consisting of those $g \in G$ such that $\Gam \cap (g \Gam g^{-1})$ has finite index in both $\Gam$ and $g \Gam g^{-1}$. Our key technical lemma is the following:

\begin{lem}\label{lem:OurTower}
Let $G$ be a group and $\Gam \le G$ a finitely generated subgroup such that $\Gam^{ab}$ has no $p$-torsion. Suppose that $\{g_0 = \Id, g_1, g_2, \dots\}$ is a sequence in $G$ such that each $g_i$ is in the commensurator of $\Gam$ in $G$. Define
\begin{align*}
\Del_i &= g_i \Gam g_i^{-1} \\
\Gam_n &= \bigcap_{i = 0}^n \Del_i.
\end{align*}
Finally, suppose:
\begin{enum}

\item The sequence $\{\Gam_n\}$ is a cofinal tower of subgroups.

\item For each $n$, there exists some $0 \le i \le n-1$ such that $\Del_i / (\Del_i \cap \Del_n)$ is an abelian $p$-group.

\end{enum}
Then $\{\Gam_n\}$ is a RFRS sequence for $\Gam$.
\end{lem}

\begin{pf}
Note that $\Del_0 = \Gam_0 = \Gam$ and
\[
\Del_0 / (\Del_0 \cap \Del_1) = \Gam_0 / \Gam_1
\]
is an abelian $p$-group. Since $\Gam_0^{ab} = \Gam^{ab}$ has no $p$-torsion, the projection from $\Gam_0$ onto $\Gam_0 / \Gam_1$ must factor through $\Gam_0^{rab}$, i.e., $(\Gam_0)_r^{(1)} \le \Gam_1$.

We now show that $(\Gam_{n+1})_r^{(1)} \le \Gam_{n+2}$ for all $n \ge 0$. Fixing $\de \in (\Gam_{n+1})_r^{(1)}$, we certainly have $\de \in \Gam_{n+1}$. Since $\Gam_{n+2} = \Gam_{n+1} \cap \Del_{n+2}$, to prove that $\{\Gam_n\}$ is a RFRS sequence, we must show that $\de \in \Del_{n+2}$.

Fix $0 \le i \le n+1$ such that $\Del_i / (\Del_i \cap \Del_{n+2})$ is an abelian $p$-group. Since $\Del_i \cong \Gam$, we see that $\Del_i^{ab}$ has no $p$-torsion, and hence $(\Del_i)_r^{(1)} \le (\Del_i \cap \Del_{n+2})$. Then $\Gam_{n+1} \le \Del_i$ by construction, and it is easy to see that the natural map $\Gam_{n+1}^{ab} \to \Del_i^{rab}$ induced by the inclusion must factor through the map from $\Gam_{n+1}$ to $\Gam_{n+1}^{rab}$. It follows that
\[
(\Gam_{n+1})_r^{(1)} \le (\Del_i)_r^{(1)} \le (\Del_i \cap \Del_{n+2}).
\]
This gives that $\de \in \Del_{n+2}$, as desired. Since $\{\Gam_n\}$ satisfies the other hypotheses to be a RFRS sequence by assumption, this completes the proof of the lemma.
\end{pf}

Our goal will be to apply Lemma \ref{lem:OurTower} to certain $\frakp$-congruence towers in arithmetic lattices. We now introduce these towers.

\subsection{$\frakp$-congruence towers}\label{ssec:Congruence}

We refer the reader to \cite[Ch.\ I and II]{Lang} for terminology and results from algebraic number theory used in this section and elsewhere in the paper. Let $k$ be a number field with integer ring $\calO_k$, $\calG \subseteq \GL_n(k)$ be a $k$-algebraic matrix group, and $\Gam = \calG(\calO_k)$. Given a prime ideal $\frakp$ of $\calO_k$ and $j \ge 1$, let $\Gam(\frakp^j)$ be the \emph{level} $\frakp^j$ \emph{congruence subgroup} of $\Gam$, i.e., all those elements that are congruent to the identity modulo $\frakp^j$. The collection $\{\Gam(\frakp^j)\}$ is the \emph{$\frakp$-congruence tower} for $\Gam$. This is a cofinal tower of normal subgroups of $\Gam$.

We record some elementary facts. Let $p$ be a rational prime. Recall that in a {\it $p$-group} every element has order a power of $p$, and in an {\it elementary $p$-group} every element has order $p$.

\begin{lem}\label{lem:pGps}
Suppose $k$ is a number field with integers $\calO_k$, $\calG \subseteq \GL_n(k)$ is a $k$-algebraic matrix group, and $\Gam = \calG(\calO_k)$. Let $\frakp$ be a prime ideal of $\calO_k$ and $p$ the characteristic of the finite field $\calO_k / \frakp$. Then:
\begin{enum}

\item For all $j \ge 1$, $\Gam(\frakp^j) / \Gam(\frakp^{j + 1})$ is an elementary abelian $p$-group.

\item For all $k > j \ge 1$, $\Gam(\frakp^j) / \Gam(\frakp^k)$ is a $p$-group.

\item For all $j \ge 2$, $\Gam(\frakp^j) / \Gam(\frakp^k)$ is abelian for every $k \le 2 j$. In particular, $\Gam(\frakp^j) / \Gam(\frakp^{j+2})$ is abelian.

\end{enum}
\end{lem}

\begin{pf}
Let $\calO_{\frakp}$ be the integral closure of $\calO_k$ in the completion $k_\frakp$ of $k$ with respect to its $\frakp$-adic norm. Fix a uniformizing element $\pi$ for $\calO_{\frakp}$. Then we have that $\calO_k / \frakp \cong \calO_{\frakp} / (\pi) \calO_{k, \frakp}$ and $p$ is the characteristic of this finite field.

If $\al \in \Gam(\frakp^j)$, then we can write
\[
\al = \Id + \pi^j M
\]
for some $M \in \M_n(\calO_{\frakp})$. Then
\[
\al^p = \sum_{k = 0}^p \binom{p}{k} \pi^{j k} M^k,
\]
which is visibly congruent to the identity modulo $\pi^{j+1}$. This proves that every element of $\Gam(\frakp^j) / \Gam(\frakp^{j+1})$ has order $p$.

Now, suppose that
\begin{align*}
\al &= \Id + \pi^j M \\
\beta &= \Id + \pi^j N
\end{align*}
for $\al, \beta \in \Gam(\frakp^j)$. Then:
\begin{align*}
\al \beta &= (\Id + \pi^j M)(\Id + \pi^j N) \\
&= \Id + \pi^j(M + N) + \pi^{2 j} M N \\
\beta \al &= (\Id + \pi^j N)(\Id + \pi^j M) \\
&= \Id + \pi^j(N + M) + \pi^{2 j} N M
\end{align*}
We see that $\al$ and $\beta$ commute modulo $\pi^k$ for all $k \le 2 j$. Since $2 j \ge j + 1$ for $j \ge 1$, w, this proves the first and third assertions of the lemma. The second statement is an immediate consequence of the first.
\end{pf}

\begin{rem}
Replacing $\calO_k$ with $\calO_{\frakp}$ in the proof of Lemma \ref{lem:pGps} is only necessary when $\frakp$ is not a principal ideal. When it is principal, one can implement the proof in $\calO_k$ instead with $\pi$ a generator for $\frakp$.
\end{rem}

\section{Examples}\label{sec:Exs}

We now describe the examples that suffice to prove the main results stated in the introduction. Our techniques work in much greater generality, and the reader will hopefully find these examples illustrative enough to apply our methods in other settings.

\subsection{The magic manifold}\label{ssec:oneBianchi}

It goes back to Thurston that the fundamental group $\Gam$ of the magic manifold arises from the congruence subgroup $\Gam(\frac{1+\sqrt{-7}}{2})$ inside $\PGL_2(\bbQ(\sqrt{-7}))$. It is homeomorphic to the complement in $S^3$ of the $3$-chain link $6^3_1$ (see Figure \ref{fig:Magic} and \cite[Ex.\ 6.8.2]{Thurston}).
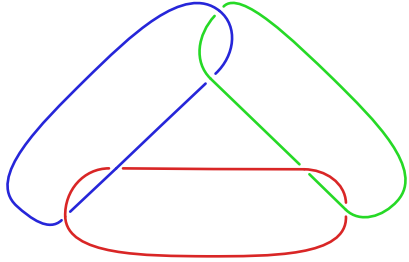
\begin{figure}
\centering
\definecolor{linkcolor0}{rgb}{0.85, 0.15, 0.15}
\definecolor{linkcolor1}{rgb}{0.15, 0.15, 0.85}
\definecolor{linkcolor2}{rgb}{0.15, 0.85, 0.15}
\begin{tikzpicture}[line width=1.0, line cap=round, line join=round,scale=0.6]
  \begin{scope}[color=linkcolor0]
    \draw (2.83, 2.06) .. controls (2.27, 2.06) and (1.84, 1.57) .. (1.86, 0.99);
    \draw (1.86, 0.99) .. controls (1.89, 0.12) and (3.62, 0.12) .. 
          (5.00, 0.11) .. controls (6.40, 0.11) and (8.11, 0.10) .. (8.08, 0.99);
    \draw (8.08, 1.30) .. controls (8.06, 1.74) and (7.63, 2.03) .. (7.16, 2.04);
    \draw (7.16, 2.04) .. controls (5.82, 2.04) and (4.48, 2.05) .. (3.14, 2.06);
  \end{scope}
  \begin{scope}[color=linkcolor1]
    \draw (5.28, 5.55) .. controls (4.50, 6.18) and (3.29, 4.99) .. 
          (2.30, 4.02) .. controls (1.31, 3.04) and (0.10, 1.86) .. 
          (0.78, 1.23) .. controls (1.10, 0.94) and (1.50, 0.65) .. (1.78, 0.91);
    \draw (1.97, 1.10) .. controls (2.31, 1.42) and (2.65, 1.74) .. (2.99, 2.06);
    \draw (2.99, 2.06) .. controls (3.65, 2.69) and (4.31, 3.32) .. (4.97, 3.94);
    \draw (5.19, 4.16) .. controls (5.62, 4.56) and (5.70, 5.21) .. (5.28, 5.55);
  \end{scope}
  \begin{scope}[color=linkcolor2]
    \draw (5.37, 5.65) .. controls (5.75, 6.04) and (6.91, 4.90) .. 
          (7.80, 4.04) .. controls (8.72, 3.13) and (9.84, 2.03) .. 
          (9.22, 1.34) .. controls (8.90, 0.98) and (8.39, 0.84) .. (8.08, 1.14);
    \draw (8.08, 1.14) .. controls (7.81, 1.40) and (7.54, 1.67) .. (7.27, 1.93);
    \draw (7.05, 2.15) .. controls (6.39, 2.78) and (5.74, 3.42) .. (5.08, 4.05);
    \draw (5.08, 4.05) .. controls (4.70, 4.42) and (4.79, 5.03) .. (5.17, 5.44);
  \end{scope}
\end{tikzpicture}
\caption{The magic manifold is the complement of the $3$-chain link.}\label{fig:Magic}
\end{figure}
Note that $\frakp = (\frac{1+\sqrt{-7}}{2})$ is a prime ideal dividing $2$. We will show that the magic manifold admits a $2$-congruence tower that is RFRS. Note that the magic manifold is itself fibered, so the fact that $\PSL_2(\calO_d)$ fibers on a congruence subgroup is not new in this case.

Since $\frakp$ has norm $2$, the completion of $\bbQ(\sqrt{-7})$ at $\frakp$ is $\bbQ_2$, and hence we obtain an embedding of $\Gam$ into $\PGL_2(\bbQ_2)$. Consider the action of $\PGL_2(\bbQ_2)$ on its Bruhat--Tits tree $\calT$, which is a $3$-regular tree (see Figure \ref{fig:Tree}).
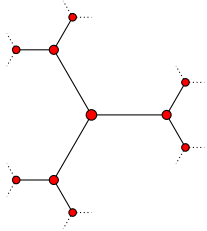
\begin{figure}
\centering
\begin{tikzpicture}
\draw (0, 0) -- (1, 0);
\draw (0, 0) -- (-0.5, 0.866025403784439);
\draw (0, 0) -- (-0.5, -0.866025403784439);
\draw (1, 0) -- (1.25, 0.433012701892219);
\draw (1, 0) -- (1.25, -0.433012701892219);
\draw (-0.5, 0.866025403784439) -- (-0.25, 1.2990381);
\draw (-0.5, 0.866025403784439) -- (-1, 0.866025);
\draw (-0.5, -0.866025403784439) -- (-1, -0.8660254);
\draw (-0.5, -0.866025403784439) -- (-0.25, -1.29903);
\draw[densely dotted] (1.25, 0.433012701892219) -- (1.5000000, 0.4330127);
\draw[densely dotted] (1.25, 0.433012701892219) -- (1.12500000, 0.64951905);
\draw[densely dotted] (1.25, -0.433012701892219) -- (1.12500000, - 0.64951905);
\draw[densely dotted] (1.25, -0.433012701892219) -- (1.5000000, - 0.4330127);
\draw[densely dotted] (-0.25, 1.2990381) -- (0, 1.29903811);
\draw[densely dotted] (-0.25, 1.2990381) -- (-0.3750000, 1.5155445);
\draw[densely dotted] (-1, 0.866025) -- (-1.1250000, 1.0825318);
\draw[densely dotted] (-1, 0.866025) -- (-1.12500000, 0.64951905);
\draw[densely dotted] (-0.25, -1.2990381) -- (-0.3750000, - 1.5155445);
\draw[densely dotted] (-0.25, -1.2990381) -- (0, - 1.29903811);
\draw[densely dotted] (-1, -0.866025) -- (-1.12500000, - 0.64951905);
\draw[densely dotted] (-1, -0.866025) -- (-1.1250000, - 1.0825318);
\draw[fill=red] (0, 0) circle [radius=.07];
\draw[fill=red] (1, 0) circle [radius=.06];
\draw[fill=red] (-0.5, 0.866025403784439) circle [radius=.06];
\draw[fill=red] (-0.5, -0.866025403784439) circle [radius=.06];
\draw[fill=red] (1.25, 0.433012701892219) circle [radius=.05];
\draw[fill=red] (1.25, -0.433012701892219) circle [radius=.05];
\draw[fill=red] (-0.25,1.2990381) circle [radius=.05];
\draw[fill=red] (-1, 0.866025) circle [radius=.05];
\draw[fill=red] (-1, -0.8660254) circle [radius=.05];
\draw[fill=red] (-0.25, -1.29903) circle [radius=.05];
\end{tikzpicture}
\caption{The Bruhat--Tits tree $\calT$ for $\PGL_2(\bbQ_2)$.} \label{fig:Tree}
\end{figure}
We briefly recall that vertices of $\calT$ are homothety classes of $\bbZ_2$-lattices in $\bbQ_2^2$, and two vertices $[\calL_1]$ and $[\calL_2]$ are adjacent if and only if there are representatives in the homothety classes such that $\calL_2 \subseteq \calL_1$ with $\calL_1 / \calL_2$ isomorphic to the finite field $\bbF_2$ with two elements. See \cite[\S II.1]{Serre} for details.

Then $\Gam \le \PGL_2(\bbZ_2)$ naturally stabilizes the vertex $v_0$ associated with the standard lattice $\bbZ_2^2$. Notice that $\PGL_2(\bbQ(\sqrt{-7}))$ acts transitively on $\calT$, and the element
\[
g_1 = \begin{pmatrix} 0 & 2 \\ 1 & 0 \end{pmatrix}
\]
exchanges $v_0$ with a neighbor $v_1$. One then checks that
\[
\Gam(\frakp^2) \le \Gam \cap g_1 \Gam g_1^{-1} \le \Gam,
\]
and it follows from Lemma \ref{lem:pGps} that $\Gam /(\Gam \cap g_1 \Gam g_1^{-1})$ is an elementary abelian $2$-group.

We now define $g_n \in \PGL_2(\bbQ(\sqrt{-7}))$ and $v_n = g_n(v_0) \in \calT$ by choosing some $v_i$ for $0 \le i \le n-1$ for which not all neighbors of $v_i$ are contained in $\{v_0, \dots, v_{n-1}\}$, letting $v_n$ be one such neighbor of $v_i$, and taking $g_n$ to be the conjugate of $g_1$ in $\PGL_2(\bbQ(\sqrt{-7}))$ that swaps $v_i$ and $v_n$. Define $\Del_n = g_n \Gam g_n^{-1}$. Then
\[
\Del_n / (\Del_n \cap \Del_i)
\]
is an elementary abelian $2$-group by the same reasoning that we applied to $\Del_0 / (\Del_1 \cap \Del_0)$.

Let $v_n$ range over all vertices of $\calT$. Defining
\[
\Gam_n = \bigcap_{i = 0}^n \Gam_i,
\]
we have that $\bigcap \Gam_n$ lies in the stabilizer in $\PGL_2(\bbQ_2)$ of every homothety class of lattices in $\bbQ_2^2$, which is clearly trivial. Therefore $\{\Gam_n\}$ is cofinal. In particular, Lemma \ref{lem:OurTower} applies to show that this tower is RFRS.

\begin{rem}
This idea applies to any principal congruence arithmetic link. In \cite{BGR} it is shown that there are principal congruence links for discriminant $d= 1,2, 3, 5, 7, 11, 15, 19, 23, 31, 47, 71$. This includes discriminants $d=7,  15, 23, 31,47 , 71$ that are congruent to $-1$ mod $8$; these values of $d$ are not handled by the next section. More generally, this construction works for congruence subgroups of arithmetic Kleinian groups with no $p$-torsion in their $1^{st}$ homology for the appropriate $p$ (see Appendix A by \c{S}eng\"un for further examples).
\end{rem}

\subsection{Bianchi groups and $\Orth(4, 1; \bbZ)$}\label{ssec:moreBianchi}

Consider the quadratic form $q_0$ in $5$ variables with matrix
\[
Q_0 = \mathrm{diag}(1, 1, 1, 1, -1),
\]
and let $\Orth(4, 1; \bbZ)$ be the group of integral automorphisms of $q_0$. Then $\Orth(4, 1; \bbZ)$ determines a nonuniform arithmetic lattice in $\Orth(4,1)$. For an integer $N \ge 1$, let $\Gam(N)$ denote the congruence subgroup of $\Orth(4, 1; \bbZ)$ of level $N$.

It is known that $\Orth(4, 1; \bbZ)$ is the group generated by reflections in the simplex in hyperbolic $4$-space with Coxeter diagram given in Figure \ref{fig:Cox4}. Moreover, the congruence subgroup $\Gam(2)$ of level two is the right-angled Coxeter group generated by reflections in the sides of a polyhedron obtained from $120$ copies of the simplex for $\Orth(4, 1; \bbZ)$. See \cite{RatcliffeTschantz}.
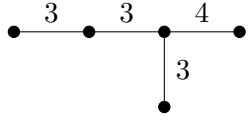
\begin{figure}
\centering
\begin{tikzpicture}
\draw (-2,0) -- (1,0);
\draw (0,0) -- (0,-1);
\draw[fill=black] (-2, 0) circle [radius=.075];
\draw[fill=black] (-1, 0) circle [radius=.075];
\draw[fill=black] (0, 0) circle [radius=.075];
\draw[fill=black] (1, 0) circle [radius=.075];
\draw[fill=black] (0, -1) circle [radius=.075];
\node at (-1.5,0.25) {$3$};
\node at (-0.5,0.25) {$3$};
\node at (0.5,0.25) {$4$};
\node at (0.25,-0.5) {$3$};
\end{tikzpicture}
\caption{The Coxeter diagram for $\Orth(4, 1; \bbZ)$.} \label{fig:Cox4}
\end{figure}

It will be convenient to change coordinates. The matrix
\[
\al = \begin{pmatrix}
1 & 1 & 0 & 0 & 0 \\
0 & -1 & 0 & 0 & 1 \\
0 & 0 & 1 & 0 & 0 \\
0 & 0 & 0 & 1 & 0 \\
1 & 1 & 0 & 0 & -1
\end{pmatrix} \in \SL_5(\bbZ)
\]
conjugates $\Orth(4, 1; \bbZ)$ to $\Orth(q; \bbZ)$, where $q$ is the quadratic form with matrix
\[
Q = \begin{pmatrix}
0 & 0 & 0 & 0 & 1 \\
0 & 1 & 0 & 0 & 0 \\
0 & 0 & 1 & 0 & 0 \\
0 & 0 & 0 & 1 & 0 \\
1 & 0 & 0 & 0 & 0
\end{pmatrix},
\]
i.e., $Q = {}^t\al Q_0 \al$. Since $\al$ is integral of determinant one, it preserves all congruence subgroups of $\SL_5(\bbZ)$. Thus $\al \Gam(N) \al^{-1}$ is the level $N$ congruence subgroup of $\Orth(q; \bbZ)$ for all $N \ge 1$, and we continue calling it just $\Gam(N)$.

We will also need the Bruhat--Tits building associated with $\Orth(q; \bbQ_2)$, which is a $(5,3)$-regular tree $\calT$. See \cite[\S2.7]{Tits} and Figure \ref{fig:2Tree}. Considering $\Orth(q; \bbQ_2)$ as a subgroup of $\GL_5(\bbQ_2)$, we obtain an injection of buildings $\calT \hookrightarrow X$, where $X$ is the building associated with $\PGL_5(\bbQ_2)$. We briefly describe $\calT$ using this embedding.
\begin{figure}
\centering
\begin{tikzpicture}
\draw (0, 0) -- (1, 0);
\draw (0, 0) -- (0.309017, 0.951057);
\draw (0, 0) -- (-0.809017, 0.587785);
\draw (0, 0) -- (-0.809017, -0.587785);
\draw (0, 0) -- (0.309017, -0.951057);
\draw (1, 0) -- (1.25, -0.433013);
\draw (1, 0) -- (1.25, 0.433013);
\draw (0.309017, 0.951057) -- (0.798091, 1.05501);
\draw (0.309017, 0.951057) -- (-0.0255483, 1.32263);
\draw (-0.809017, 0.587785) -- (-0.756753, 1.08505);
\draw (-0.809017, 0.587785) -- (-1.26579, 0.384417);
\draw (-0.809017, -0.587785) -- (-1.26579, - 0.384417);
\draw (-0.809017, -0.587785) -- (-0.756753, - 1.08505);
\draw (0.309017, -0.951057) -- (-0.0255483, - 1.32263);
\draw (0.309017, -0.951057) -- (0.798091, - 1.05501);
\draw[densely dotted] (1.25, -0.433013) -- (1.00546310, - 0.4849906);
\draw[densely dotted] (1.25, -0.433013) -- (1.22386788, - 0.68164318);
\draw[densely dotted] (1.25, -0.433013) -- (1.4783864, - 0.5346969);
\draw[densely dotted] (1.25, -0.433013) -- (1.4172827, - 0.2472265);
\draw[densely dotted] (1.25, 0.433013) -- (1.4172827, 0.2472265);
\draw[densely dotted] (1.25, 0.433013) -- (1.4783864, 0.5346969);
\draw[densely dotted] (1.25, 0.433013) -- (1.22386788, 0.68164318);
\draw[densely dotted] (1.25, 0.433013) -- (1.00546310, 0.48499062);
\draw[densely dotted] (0.798091, 1.05501) -- (0.77195868, 0.80638189);
\draw[densely dotted] (0.798091, 1.05501) -- (1.02647716, 0.95332820);
\draw[densely dotted] (0.798091, 1.05501) -- (0.9653734, 1.2407986);
\draw[densely dotted] (0.798091, 1.05501) -- (0.6730908, 1.2715187);
\draw[densely dotted] (-0.0255483, 1.32263) -- (0.2028381, 1.4243131);
\draw[densely dotted] (-0.0255483, 1.32263) -- (-0.0516804, 1.5712594);
\draw[densely dotted] (-0.0255483, 1.32263) -- (-0.27008521, 1.37460685);
\draw[densely dotted] (-0.0255483, 1.32263) -- (-0.15054831, 1.10612258);
\draw[densely dotted] (-0.756753, 1.08505) -- (-0.52836640, 0.98336204);
\draw[densely dotted] (-0.756753, 1.08505) -- (-0.58947011, 1.27083241);
\draw[densely dotted] (-0.756753, 1.08505) -- (-0.8817528, 1.3015526);
\draw[densely dotted] (-0.756753, 1.08505) -- (-1.0012897, 1.0330683);
\draw[densely dotted] (-1.26579, 0.384417) -- (-1.2919218, 0.6330474);
\draw[densely dotted] (-1.26579, 0.384417) -- (-1.5103266, 0.4363949);
\draw[densely dotted] (-1.26579, 0.384417) -- (-1.39078972, 0.16791058);
\draw[densely dotted] (-1.26579, 0.384417) -- (-1.09850707, 0.19863072);
\draw[densely dotted] (-1.26579, - 0.384417) -- (-1.09850707, - 0.19863072);
\draw[densely dotted] (-1.26579, - 0.384417) -- (-1.39078972, - 0.16791058);
\draw[densely dotted] (-1.26579, - 0.384417) -- (-1.5103266, - 0.4363949);
\draw[densely dotted] (-1.26579, - 0.384417) -- (-1.2919218, - 0.6330474);
\draw[densely dotted] (-0.756753, - 1.08505) -- (-1.0012897, - 1.0330683);
\draw[densely dotted] (-0.756753, - 1.08505) -- (-0.8817528, - 1.3015526);
\draw[densely dotted] (-0.756753, - 1.08505) -- (-0.58947011, - 1.27083241);
\draw[densely dotted] (-0.756753, - 1.08505) -- (-0.52836640, - 0.98336204);
\draw[densely dotted] (-0.0255483, - 1.32263) -- (-0.15054831, - 1.10612258);
\draw[densely dotted] (-0.0255483, - 1.32263) -- (-0.27008521, - 1.37460685);
\draw[densely dotted] (-0.0255483, - 1.32263) -- (-0.0516804, - 1.5712594);
\draw[densely dotted] (-0.0255483, - 1.32263) -- (0.2028381, - 1.4243131);
\draw[densely dotted] (0.798091, - 1.05501) -- (0.6730908, - 1.2715187);
\draw[densely dotted] (0.798091, - 1.05501) -- (0.9653734, - 1.2407986);
\draw[densely dotted] (0.798091, - 1.05501) -- (1.02647716, - 0.95332820);
\draw[densely dotted] (0.798091, - 1.05501) -- (0.77195868, - 0.80638189);
\draw[fill=red] (0, 0) circle [radius=.07];
\draw[fill=blue] (1, 0) circle [radius=.06];
\draw[fill=blue] (0.309017, 0.951057) circle [radius=.06];
\draw[fill=blue] (-0.809017, 0.587785) circle [radius=.06];
\draw[fill=blue] (-0.809017, -0.587785) circle [radius=.06];
\draw[fill=blue] (0.309017, -0.951057) circle [radius=.06];
\draw[fill=red] (1.25, -0.433013) circle [radius=.05];
\draw[fill=red] (1.25, 0.433013) circle [radius=.05];
\draw[fill=red] (0.798091, 1.05501) circle [radius=.05];
\draw[fill=red] (-0.0255483, 1.32263) circle [radius=.05];
\draw[fill=red] (-0.756753, 1.08505) circle [radius=.05];
\draw[fill=red] (-1.26579, 0.384417) circle [radius=.05];
\draw[fill=red] (-1.26579, - 0.384417) circle [radius=.05];
\draw[fill=red] (-0.756753, - 1.08505) circle [radius=.05];
\draw[fill=red] (-0.0255483, - 1.32263) circle [radius=.05];
\draw[fill=red] (0.798091, - 1.05501) circle [radius=.05];
\end{tikzpicture}
\caption{The Bruhat--Tits tree $\calT$ for $\Orth(q; \bbQ_2)$.} \label{fig:2Tree}
\end{figure}
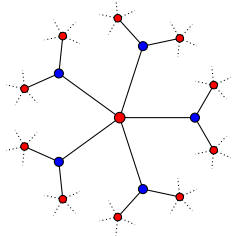

As in the $2$-dimensional case, vertices of $X$ are in one-to-one correspondence with homothety classes of lattices in $\bbQ_2^5$, where vertices $x$ and $y$ are adjacent if there are representatives $\calL_x$ and $\calL_y$ for the two homothety classes so that $2 \calL_x \subset \calL_y \subset \calL_x$. See \cite[Exer.\ II.1.4]{Serre}. We fix the base vertex
\[
x_0 = [\langle e_1, \dots, e_5 \rangle ],
\]
where $\{e_i\}$ is the basis for which $q$ has the given matrix and $\langle\,-\,\rangle$ denotes the $\bbZ_2$-span. The stabilizer of $x_0$ in $\GL_5(\bbQ_2)$ is generated by $\GL_5(\bbZ_2)$ and the scalar matrices. Then $\Orth(q; \bbZ_2)$ stabilizes $x_0$, which implies that we can realize $\calT$ as the convex hull of the $\Orth(q; \bbQ_2)$-orbit in $X$ of $x_0$.

The apartment $A$ of $X$ associated with the $\bbQ_2$-split torus of diagonal matrices in $\GL_5(\bbQ_2)$ (e.g., see \cite[\S 1]{Tits}) can be identified with the set of homothety classes
\[
\left\{ [ \langle 2^{r_1} e_1, \dots, 2^{r_5} e_5 \rangle ] \right\}_{r_i \in \bbZ},
\]
i.e., the orbit of $x_0$ under the diagonal subgroup. The $\bbQ_2$-split torus of diagonal matrices in the rank one group $\Orth(q; \bbQ_2)$ is
\[
S = \left\{ \begin{pmatrix} \lam & 0 & 0 \\ 0 & \Id & 0 \\ 0 & 0 & \lam^{-1} \end{pmatrix}\ :\ \lam \in \bbQ_2^* \right\},
\]
where $\Id$ is the $3 \times 3$ identity matrix, and the convex hull of the $S$-orbit of $x_0$ is then the apartment $A_0$ of $\calT$ associated with $S$. The $S$-orbit of $x_0$ is
\[
x_r = [\langle 2^r e_1, e_2, \dots, e_4, 2^{-r} e_5 \rangle ]
\]
for $r \in \bbZ$, and its convex hull also includes the vertices
\[
x_{r + \frac{1}{2}} = [\langle 2^{r+1} e_1, e_2, \dots, e_4, 2^{-r} e_5 \rangle ].
\]
We see that $A_0$ is a line with vertex set $\{x_\al\ :\ \al \in \frac{1}{2} \bbZ \}$, where $x_\al$ is adjacent to $x_\beta$ if and only if $|\al - \beta| = \frac{1}{2}$.

Since $\Orth(q; \bbQ_2)$ acts transitively on apartments of $\calT$ \cite[\S 2]{Tits}, the vertex set of $\calT$ is the $\Orth(q; \bbQ_2)$-orbit of $\{x_0, x_{\frac{1}{2}}\}$. In particular, to prove that $\calT$ is a $(5,3)$-regular tree we need to prove the following two lemmas.

\begin{lem}\label{lem:Valence5}
The vertex $x_0 \in \calT$ has valence $5$.
\end{lem}

\begin{lem}\label{lem:Valence3}
The vertex $x_{\frac{1}{2}} \in \calT$ has valence $3$.
\end{lem}

\begin{pf}[Proof of Lemma \ref{lem:Valence5}]
We must compute the $\Orth(q; \bbZ_2)$-orbit of $x_{\frac{1}{2}}$. We define $\calL_0 = \langle e_1, \dots, e_5 \rangle$ and $\calL_{\frac{1}{2}} = \langle 2 e_1, e_2, \dots, e_5 \rangle$. Neighbors of $x_0$ in $X$ are in one-to-one correspondence with proper nonzero subspaces of
\[
V_0 = \calL_0 / 2 \calL_0 \cong \bbF_2^5.
\]
Let $\{\conj{e}_i\}$ be the basis for $V_0$ induced by $\{e_i\}$.

If $\conj{q}_0$ denotes the quadratic form on $V_0$ induced by the restriction of $q$ to $\calL_0$, then we see that the image of $\calL_{\frac{1}{2}}$ in $V_0$ is the $\conj{q}_0$-orthogonal complement $\conj{e}_5^\perp$ of $\conj{e}_5$, which we note is a codimension one subspace that contains the isotropic vector $\conj{e}_5$. To prove the lemma it then suffices to compute the orbit of $\conj{e}_5^\perp$ under the image $G_0$ of $\Orth(q; \bbZ_2)$ under reduction modulo $2$ (e.g., see \cite[\S3.5.4]{Tits}).

Since $v^\perp = \conj{e}_5^\perp$ if and only if $v = \conj{e}_5$, it moreover suffices to compute the $G_0$-orbit of $\conj{e}_5$. One checks that this orbit is
\[
\{ \conj{e}_5\,,\, \conj{e}_1\,,\, \conj{e}_1 + \conj{e}_2 + \conj{e}_3 + \conj{e}_5\,,\, \conj{e}_1 + \conj{e}_2 + \conj{e}_4 + \conj{e}_5\,,\, \conj{e}_1 + \conj{e}_3 + \conj{e}_4 + \conj{e}_5\}.
\]
This proves the lemma.
\end{pf}

\begin{pf}[Proof of Lemma \ref{lem:Valence3}]
The proof is very similar to the proof of Lemma \ref{lem:Valence5}, so we sketch the argument and leave it to the reader to verify the details. With notation as in that proof, we consider
\[
2 \calL_{\frac{1}{2}} \subset 2 \calL_0 \subset \calL_{\frac{1}{2}} = \langle f_1, \dots, f_5 \rangle.
\]
Then $V_{\frac{1}{2}} = \calL_{\frac{1}{2}} / 2 \calL{\frac{1}{2}}$ is a vector space with basis $\{\conj{f}_i\}$ with respect to which the quadratic form $\conj{q}_{\frac{1}{2}}$ contains a two-dimensional totally degenerate subspace spanned by $\conj{f}_1$ and $\conj{f}_5$.

Since the image of $2 \calL_0$ in $V_{\frac{1}{2}}$ is the line spanned by $\conj{f}_1$, we must compute its orbit under the reduction modulo $2$ of the stabilizer in $\Orth(q; \bbQ_2)$ of $\calL_{\frac{1}{2}}$. One shows that this orbit consists of the lines spanned by $\conj{f}_1$, $\conj{f}_5$, and $\conj{f}_1 + \conj{f}_5$, and the lemma follows.
\end{pf}

\begin{rem}
The neighbors of $x_0$ in $A_0$ are $x_{\frac{1}{2}}$ and $x_{-\frac{1}{2}}$. The other neighbors have representatives:
\begin{align*}
&\langle 2 e_1\,,\, e_1 + e_2\,,\, e_1 + e_3\,,\, e_4\,,\, e_1+e_2+e_3+e_5 \rangle \\
&\langle 2 e_1\,,\, e_1 + e_2\,,\, e_3\,,\, e_1 + e_4\,,\, e_1+e_2+e_4+e_5 \rangle \\
&\langle 2 e_1\,,\, e_2\,,\, e_1 + e_3\,,\, e_1 + e_4\,,\, e_1+e_3+e_4+e_5 \rangle
\end{align*}
Similarly, $x_{\frac{1}{2}}$ has neighbors $x_0$ and $x_1$ along with the vertex with representative
\[
\langle e_1 + \frac{1}{2} e_5\,,\, e_2\,,\, e_3\,,\, e_4\,,\, e_5 \rangle,
\]
which is the image of $\calL_0$ under the matrix
\[
\begin{pmatrix}
1 & 0 & 0 & 0 & 0 \\
1 & 1 & 0 & 0 & 0 \\
0 & 0 & 1 & 0 & 0 \\
0 & 0 & 0 & 1 & 0 \\
-\frac{1}{2} & -1 & 0 & 0 & 1
\end{pmatrix} \in \SO(q; \bbQ_2)
\]
that stabilizes $\calL_{\frac{1}{2}}$.
\end{rem}

We are now ready to prove the main result of this section. Recall that the Bianchi groups $\PSL_2(\calO_d)$ with $d \not\equiv -1 \pmod{8}$ are all commensurable with subgroups of $\Orth(4, 1; \bbZ)$, and one can choose the subgroup of $\PSL_2(\calO_d)$ contained in $\Orth(4,1; \bbZ)$ to be a congruence subgroup of each. We can then use the following to prove Theorem \ref{thm:Bianchi}.

\begin{prop}\label{prop:Cox4}
The congruence subgroup $\Gam(4)$ of level $4$ in $\Orth(4, 1; \bbZ)$ admits a congruence RFRS tower.
\end{prop}

\begin{pf}
Using the presentation for $\Orth(4,1;\bbZ)$ as a Coxeter group with diagram as in Figure \ref{fig:Cox4}, one can easily check using a computer algebra program like Magma \cite{Magma} that $\Gam(4)^{ab} \cong \bbZ^{55}$. In particular, it has no $2$-torsion.

Consider the matrix
\[
g_1 = \begin{pmatrix}
0 & 0 & 0 & 0 & 2 \\
0 & -1 & 0 & 0 & 0 \\
0 & 0 & -1 & 0 & 0 \\
0 & 0 & 0 & -1 & 0 \\
\frac{1}{2} & 0 & 0 & 0 & 0
\end{pmatrix} \in \SO(q; \bbQ) < \SO(q; \bbQ_2)
\]
that exchanges the vertices $x_0, x_1 \in \calT$ and fixes the intermediate vertex $x_{\frac{1}{2}}$.

We now set $\Gam_0 = \Gam(4)$ and $\Gam_1 = g_1 \Gam_0 g_1^{-1}$. We claim that
\[
\Gam(16) \le \Gam_0 \cap \Gam_1 \le \Gam_0.
\]
To see this, one first notices that
\begin{align*}
g_1 &\begin{pmatrix}
1 + 16 a_5 & -32 e_2 & -32 e_3 & -32 e_4 & 64 e_1 \\
-8 b_5 & 1 + 16 b_2 & 16 b_3 & 16 b_4 & -32 b_1 \\
-8 c_5 & 16 c_2 & 1 + 16 c_3 & 16 c_4 & -32 c_1 \\
-8 d_5 & 16 d_2 & 16 d_3 & 1 + 16 d_4 & -32 d_1 \\
4 a_1 & -8 a_2 & -8 a_3 & -8 a_4 & 1 + 16 a_1
\end{pmatrix} g_1^{-1} \\
= &\begin{pmatrix}
1 + 16 a_1 & 16 a_2 & 16 a_3 & 16 a_4 & 16 a_5 \\
16 b_1 & 1 + 16 b_2 & 16 b_3 & 16 b_4 & 16 b_5 \\
16 c_1 & 16 c_2 & 1 + 16 c_3 & 16 c_4 & 16 c_5 \\
16 d_5 & 16 d_2 & 16 d_3 & 1 + 16 d_4 & 16 d_5 \\
16 e_1 & 16 e_2 & 16 e_3 & 16 e_4 & 1 + 16 e_5
\end{pmatrix}
\end{align*}
for $a_1, \dots, e_5 \in \bbZ$. Also, since $g_1 \in \SO(q, \bbQ)$, the matrix on the right-hand side preserves $q$ if and only if the matrix being conjugated on the left-hand side does. This proves the claim.

Lemma \ref{lem:pGps} implies that $\Gam(4) / \Gam(16)$ is an abelian $2$-group, hence so is $\Gam_0 / \Gam_1$. We now define $g_n$ inductively as follows. Let $x_n$ be a vertex of $\calT$ in the $\SO(q; \bbQ_2)$-orbit of $x_0$ that is distance $2$ in $\calT$ from some vertex $x_i$ in $\{x_0, \dots, x_{n-1}\}$. Since $\SO(q; \bbQ)$ is dense in $\SO(q; \bbQ_2)$, there exists an $h_n$ in $\SO(q; \bbQ)$ so that $h_n(x_0) = x_i$ and $h_n(x_1) = x_n$. We define $g_n = h_n g_1 h_n^{-1}$.

Choose the sequence $\{x_n\}$ to exhaust the $\SO(q; \bbQ_2)$-orbit of $x_0$. Then it is easy to see that the sequence $\{g_n\}$ satisfies all the conditions of Lemma \ref{lem:OurTower}. In particular, if $\Del_n = g_n \Gam_0 g_n^{-1}$ and
\[
\Gam_n = \bigcap_{i = 0}^n \Del_n,
\]
then $\{\Gam_n\}$ is cofinal, and the elements $g_n$ satisfy the requisite assumptions by construction. Therefore there is a RFRS tower for $\Gam(4)$.
\end{pf}

\begin{cor}\label{cor:SOfiber}
The group $\SO(4,1;\mathbb{Z})$ has a congruence subgroup that is algebraically fibered.
\end{cor}
\begin{proof}
Since $\SO(4,1;\mathbb{Z})$ has $b_1^{(2)}=0$ (see \cite[Lem.\ 1]{Lott}), a result of Dawid Kielak \cite[Thm.\ 5.3]{Kielak} implies that some level $2^k$ congruence subgroup is algebraically fibered.
\end{proof}

\begin{rem}
The proof of Proposition \ref{prop:Cox4} would work without alteration for the level $4$ congruence subgroup of $\SO(n,1; \bbZ)$ for any $n \ge 2$, as long as it has no $2$-torsion in its abelianization. This holds for $n = 2, 3, 4$, however Steven Tschantz computed that $H_1(\Gamma(4);\bbZ)\cong \bbZ^{256}\times \bbZ/2$ for $\Gamma=\SO(5,1;\bbZ)$. 
\end{rem}

\begin{rem}
For $n = 2, \dots, 7$, the congruence subgroup $\Gam(2)$ of level $2$ in $\Orth(n, 1; \bbZ)$ is a right-angled Coxeter group \cite[Thm.\ 7]{RatcliffeTschantz}. The commutator subgroup $\Gam(2)^{(1)}$ then has torsion-free abelianization \cite[\S4.5]{BV}. For $n \le 4$, $\Gam(2)^{(1)}$ equals $\Gam(4)$.  For $n > 4$, we have that $\Gam(2)^{(1)}$ is a proper finite index subgroup of $\Gam(4)$. It is possible that $H_1(\Gam(4))$ is torsion-free for $n>5$ in spite of Tschantz's computation for $n=5$.
\end{rem}

\begin{pf}[Proof of Theorem \ref{thm:Bianchi}]
For $d \not\equiv -1 \pmod{8}$, there is a finite index subgroup $\Del$ of $\PSL_2(\calO_d)$ that is isomorphic to a subgroup of the group $\Gam(4)$ in Proposition \ref{prop:Cox4}. For $d$ square-free, the quadratic form $q_d =\langle 1,1,1, -d\rangle$ is isotropic if and only if $d\not\equiv -1 \pmod{8}$; \cite[Thm.\ 6.2]{ALR} and the subsequent discussion. In this case, $\PO(q_d;\mathbb{Z})$ is commensurable with $\PSL_2(\calO_d)$ \cite[Thm.\ 2.3]{ALR}, hence one can embed $\PO(q_d;\bbZ)$ into $\PO(4,1;\bbZ)$ up to commensurability by \cite[Lem.\ 6.3]{ALR} and as in the proof of \cite[Lem.\ 4.6(i)]{ALR}. The proof that such a $\Del$ exists in fact produces a congruence subgroup of $\PSL_2(\calO_d)$. Intersecting this with the RFRS tower given in Proposition \ref{prop:Cox4} produces the desired RFRS tower for $\Del$. This proves the theorem.
\end{pf}

\begin{rem}
Methods analogous to work of Michelle Chu \cite{Chu} on effectively embedding subgroups of Bianchi groups in $\SO(6,1; \bbZ)$ could allow one to prove that the above Bianchi groups contain a congruence RFRS tower of uniformly bounded index.
\end{rem}

\begin{rem}
If the congruence subgroup of level $4$ in $\SO(6,1; \bbZ)$ has no $2$-torsion in its first homology, then Theorem \ref{thm:Bianchi} holds for all Bianchi groups. See \cite[Lem.\ 4.4]{ALR}. More generally, one only needs to find a prime $p$ so that the congruence subgroup of level $p$ in $\SO(6, 1; \bbZ)$ has no $p$-torsion in its abelianization, which seems likely but very difficult to verify computationally. If this holds, then all Bianchi groups contain a congruence RFRS tower and hence fiber on a congruence subgroup.
\end{rem}

\subsection{A complex hyperbolic example}\label{ssec:CH2}

Our example will come from a congruence cover of a Deligne--Mostow orbifold \cite{DeligneMostow}. We recall that for certain $(n+3)$-tuples $\mu$ of integers (called \emph{weights}) satisfying a condition called INT, Deligne and Mostow constructed lattices $\Gam_{\mu} \trianglelefteq \Gam_{\Sig \mu}$ in $\PU(n,1)$, where $\Sig$ is the symmetry group of the weights and $\Gam_{\Sig \mu} / \Gam_{\mu} \cong \Sig$. Let $\bbB^2$ denote complex hyperbolic $2$-space in what follows.

The example we consider here is $\mu = (2,2,2,2,2)$, hence $\Sig = S_5$. Following \cite{KirwanLeeWeintraub}, the underlying analytic space for the orbifold $\bbB^2 / \Gam_\mu$ is the blowup of the complex projective plane $\bbP^2$ at the four vertices of the complete quadrangle, and each divisor has orbifold weight $5$. See Figure \ref{fig:DM}. (Note that the convention in \cite{KirwanLeeWeintraub} is to divide the elements of $\mu$ by their gcd, so $\mu$ is listed as $(1,1,1,1,1)$.) Then $S_5$ acts on this blowup of $\bbP^2$ in a natural way with quotient the underlying analytic space for $\bbB^2 / \Gam_{\Sig \mu}$.
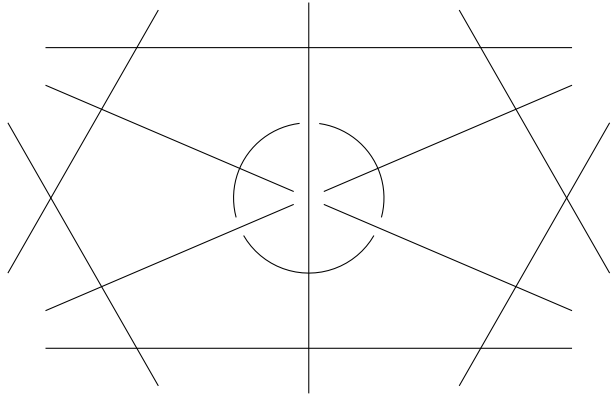
\begin{figure}
\centering
\begin{tikzpicture}[scale=2]
\draw (-1.75,1) -- (1.75,1);
\draw (-1.75,-1) -- (1.75,-1);
\draw (-1,-1.25) -- (-2,0.5);
\draw (1,-1.25) -- (2,0.5);
\draw (-1,1.25) -- (-2,-0.5);
\draw (1,1.25) -- (2,-0.5);
\draw (0,-1.3) -- (0,1.3);
\draw [domain=210:330] plot ({0.5*cos(\x)}, {0.5*sin(\x)});
\draw [domain=97:195] plot ({0.5*cos(\x)}, {0.5*sin(\x)});
\draw [domain=-15:82] plot ({0.5*cos(\x)}, {0.5*sin(\x)});
\draw (0.1, 0.04285714285) -- (1.75,0.75);
\draw (-0.1, -0.04285714285) -- (-1.75,-0.75);
\draw (-0.1, 0.04285714285) -- (-1.75,0.75);
\draw (0.1, -0.04285714285) -- (1.75,-0.75);
\end{tikzpicture}
\caption{The orbifold $\bbB^2 / \Gam_\mu$. Each line or circle represents a $\bbP^1$ in the orbifold locus, and each has orbifold weight $5$. Local orbifold groups at intersection points are all $(\bbZ / 5)^2$.}\label{fig:DM}
\end{figure}

It is known that these lattices are arithmetic. More specifically, let $E = \bbQ(\zeta_5)$, where $\zeta_5$ is a primitive $5^{th}$ root of unity, and $F = \bbQ(\al)$ with $\al^2 = 5$ be its totally real quadratic subfield. Define $\phi = \frac{1 - \al}{2}$ and consider the hermitian form on $E^3$ with matrix
\[
h = \begin{pmatrix} \phi & 1 & 0 \\ 1 & \phi  & 1 \\ 0 & 1 & \phi \end{pmatrix}.
\]
Then $h$ has signature $(2,1)$ at one complex place of $E$ and signature $(3,0)$ at the other complex place. It follows that $\PU(h, \calO_E)$ is a cocompact arithmetic lattice in $\PU(2,1)$, where $\calO_E = \bbZ[\zeta_5]$ is the ring of integers of $E$. Let $\pi = \zeta_5-1$ and $\frakp_5 = \pi \calO_E$ be the unique prime ideal of $\calO_E$ dividing $5 \calO_E$. Note that $\calO_E / \frakp_5 \cong \bbF_5$, $\frakp_5^2 = \al \calO_E$, and $\frakp_5^4 = 5 \calO_E$. We then have the following.

\begin{prop}\label{prop:5Congruence}
With notation as above, $\Gam_{\Sig \mu} \cong \PU(h, \calO_E)$ and $\Gam_\mu$ is the congruence subgroup $\Gam_{\Sig \mu}(\frakp_5)$ of level $\frakp_5$ in $\Gam_{\Sig \mu}$.
\end{prop}

\begin{pf}
It follows from Deligne and Mostow's work that $h$ is the hermitian form associated with $\mu$, hence $\Gam_{\Sig \mu} \le \PU(h, \calO_E)$. For example, see \cite[Prop.\ 4.6]{Looijenga} (technically, this gives a scalar multiple of $h$, but similar hermitian forms have the same unitary group). To show that the two are equal, it suffices to know that $\bbB^2 / \Gam_{\Sig\mu}$ and $\bbB^2 / \PU(h, \calO_E)$ have the same orbifold Euler characteristic. See \cite[Table 1]{McMullen} for $\Gam_{\Sig \mu}$ and \cite[\S 8.2]{PrasadYeung} for $\PU(h, \calO_E)$ (noting that $\mu$ in \cite{PrasadYeung} is one-third the orbifold Euler characteristic and that our arithmetic lattice comes from their case $\calC_1$). This proves the first part of the proposition.

Note that $h$ becomes a nondegenerate \emph{quadratic} form modulo $\frakp_5$, hence reduction modulo $\frakp_5$ is a surjective homomorphism $\Gam_{\Sig \mu} \to \PO(3,5) \cong S_5$. We also have the surjective homomorphism $\Gam_{\Sig \mu} \to S_5$ with kernel $\Gam_\mu$ defined by the automorphisms of $\mu$. We want to know that these homomorphisms are the same. However, using one of the known presentations for $\Gam_{\Sig\mu}$ (e.g., see \cite[\S 7]{CSmatrix} or \cite{DPP}) one easily checks with Magma \cite{Magma} that there is a unique homomorphism from $\Gam_{\Sig\mu}$ onto $\PO(3,5)$, hence $\Gam_\mu$ must be the congruence subgroup of level $\frakp_5$, and this completes the proof of the proposition.
\end{pf}

\begin{rem}
One can likely show without appealing to the computer that $\Gam_\mu = \Gam_{\Sig \mu}(\frakp_5)$ by showing that the order $5$ reflection generators of $\Gam_\mu$ associated with allowing two weights to coalesce all have $\zeta_5$-eigenline that reduce to a nondegenerate line modulo $\frakp_5$. Since $\bbF_5$ has no nontrivial $5^{th}$ roots of unity, each generator is trivial modulo $\frakp_5$, hence $\Gam_\mu \le \Gam_{\Sig \mu}(\frakp_5)$. The two groups have the same index in $\Gam_{\Sig \mu}$, hence they are the same.
\end{rem}

\noindent
The lattice of interest to us is supplied by the following lemma.

\begin{lem}\label{lem:OurCHgroup}
With notation as above, the congruence subgroup of level $\frakp_5^2 = \al \calO_E$ in $\Gam_{\Sig \mu}$ is the commutator subgroup of $\Gam_\mu$ and its abelianization is isomorphic to $\bbZ^{60}$.
\end{lem}

\begin{pf}
We have that $\Gam_\mu = \Gam_{\Sig\mu}(\frakp_5)$ is the congruence subgroup of $\Gam_{\Sig\mu}$ of level $\frakp_5$ by Proposition \ref{prop:5Congruence}. A direct calculation shows that there are exactly $5^5$ elements of $\PU(h, \calO_E / \frakp_5^2)$ that are congruent to the identity modulo $\frakp_5$. To see this, write $\calO_E / \frakp_5^2$ as $\bbF_5[\ep]$ for $\ep^5 = 1$ and $(\ep-1)^2 = 0$, where the Galois involution of $\calO_E$ descends to the automorphism of $\calO_E / \frakp_5^2$ generated by $\ep \mapsto -\ep$. The ideal $(\ep)$ is the kernel of the natural reduction map from $\calO_E / \frakp_5^2$ to $\bbF_5$, and this ideal consists of the multiples of $\ep$ by $0,\dots, 4 \in \bbF_5$.

The image of $h$ in $\SL_3(\calO_E / \frakp_5^2)$ has matrix
\[
\begin{pmatrix} 3 & 1 & 0 \\ 1 & 3 & 1 \\ 0 & 1 & 3 \end{pmatrix}.
\]
One then calculates that lifts to $\GL_3(\calO_E / \frakp_5^2)$ of elements of $\PU(h, \calO_E / \frakp_5^2)$ that are congruent to the identity modulo $(\ep)$ are matrices of the form
\[
\begin{pmatrix}
1+ \ep x_1 & \ep x_2 & \ep x_3 \\
\ep y_1 & 1 + \ep y_2 & \ep y_3 \\
\ep z_1 & \ep z_2 & 1 + \ep z_3
\end{pmatrix}
\]
with $x_1, \dots, z_3 \in \{0, \dots, 4\} \subset \bbF_5 \subset \calO_E / \frakp_5^2$ satisfying the three linearly independent equations:
\begin{align*}
x_1 - 3 x_2 + 3 y_1 - y_2 + z_1 &= \\
3 x_3 - y_1 + y_3 - 3 z_1 &= \\
x_3 - y_2 + 3 y_3 - 3 z_2 + z_3 &= 0
\end{align*}
There are $5^6$ such matrices, and exactly $5$ of them are scalar matrices. Therefore projectivization leaves us with $5^5$ possibilities.

Lemma \ref{lem:pGps} then implies that $\Gam_{\Sig \mu}(\frakp_5) / \Gam_{\Sig\mu}(\frakp_5^2)$ is isomorphic to $(\bbZ / 5)^5$. Computing the abelianization of $\Gam_\mu$ (either from a presentation or geometrically from the orbifold), one sees that $\Gam_\mu^{ab} \cong (\bbZ / 5)^5$, hence $\Gam_{\Sig\mu}(\frakp_5^2)$ is the commutator subgroup of $\Gam_\mu$. One then computes the abelianization of the commutator subgroup of $\Gam_\mu$ in Magma \cite{Magma} to complete the proof of the lemma.
\end{pf}

\begin{rem}
We note the following analogy between $\Orth(4,1; \bbZ)$ and $\Gam_{\Sig\mu}$. Recall that $\Orth(4,1; \bbZ)$ is a Coxeter group whose congruence subgroup of level $2$ is a right-angled Coxeter group, and the commutator subgroup of the right-angled group is the congruence subgroup of level $4$ in $\Orth(4,1;\bbZ)$. Analogously, $\Gam_{\Sig\mu}$ is a complex hyperbolic reflection group whose congruence subgroup of level $\frakp_5$ is the complex hyperbolic reflection group $\Gam_\mu$, and the commutator subgroup of $\Gam_\mu$ is the congruence subgroup of $\Gam_{\Sig\mu}$ of level $\frakp_5^2$.
\end{rem}

We now describe the building used to apply our methods to prove that $\Gam_{\Sig\mu}(\frakp_5^2)$ admits a congruence RFRS tower. Let $E_5$ be the completion of $E$ with respect to the valuation associated with $\frakp_5$. Then $E_5 = \bbQ_5(\zeta_5)$ is a degree four totally ramified extension of $\bbQ_5$ with intermediate quadratic subfield $F_5 = \bbQ_5(\al)$ and $\pi$ is a uniformizer for $E_5$. The group $\SU(h, E_3)$ is the unique special unitary group in $3$ variables with respect to $E_5 / F_5$, and the associated Bruhat--Tits building is a tree \cite[\S2.10]{Tits}.

As in \S\ref{ssec:moreBianchi}, a change of coordinates will be convenient for describing this building. One can find a change of coordinates with entries in the ring of integers $\calO_5$ of $E_5$ so that $h$ has matrix $-h_0$ for
\[
h_0 = \begin{pmatrix} 0 & 0 & 1 \\ 0 & 1 & 0 \\ 1 & 0 & 0 \end{pmatrix}.
\]
For example,
\[
c = \begin{pmatrix}
1 & -\de^{-1} & \frac{-1+\al}{8} \\
-1 + \de & \de^{-1} & \frac{(1-\al)(1+\de)}{8} \\
-1-\al+\ep & 0 & -\frac{2+\de}{4}
\end{pmatrix}
\]
suffices, where $\de$ is a square root of $1+\al$ and $\ep$ is a square root of $4 + 2 \al$. Critically, $\de, \ep \in \calO_5^*$ (one can see this by showing that the prime $\al \calO_F$ of $\calO_F$ dividing $5$ splits in both $F(\de)$ and $F(\ep)$, and $\de,\ep$ are invertible in $\calO_5^*$ since $1+\al$ and $4+2 \al$ have norm $-4$). Since this conjugation is integral over $E_5$ with determinant $1$, and because similar hermitian forms have the same unitary group, we have that $\Gam_{\Sig\mu}$ is isomorphic to the intersection of the $F$-points of $\SU(h_0)$ with $\SU(h_0, \calO_5)$.

Following \cite[\S2.10]{Tits}, the building for $\SU(h_0, E_5)$ has vertices the set of additive norms $\phi$ on $E_5^3$ so that
\[
\nu(h_0(x,y)) \ge \phi(x) + \phi(y)
\]
for all pairs $x,y \in E_5^3$, where $\nu$ is the extension to $E_5$ of the normalized valuation on $F_5$ (i.e., with value group $\frac{1}{2}\bbZ$). There is an obvious action of $\SU(h_0)$ on the set of norms, and the norm stabilized by $\SU(h_0, \calO_5)$ is the vertex $v_0$ associated with the norm
\[
\phi_0(x_1, x_2, x_3) = \inf\{\nu(x_j)\ :\ 1 \le j \le 3 \}.
\]
In particular, $\Gam_{\Sig\mu}$ stabilizes this vertex.

The matrix
\[
g_0 = \begin{pmatrix} 0 & 0 & \pi \\ 0 & \zeta_5^4 & 0 \\ \conj{\pi}^{-1} & 0 & 0 \end{pmatrix} \in \SU(h)
\]
(where $\conj{\pi}$ is the conjugate of $\pi$ for the $\Gal(E_5 / F_5)$-action) acts on the tree by sending $v_0$ to the vertex $v_1$ associated with the norm
\[
\phi_1(x_1, x_2, x_3) = \inf\left\{\nu(x_1) - \frac{1}{2}, \nu(x_2), \nu(x_3) + \frac{1}{2}\right\},
\]
since $\nu(\pi) = \frac{1}{2}$. It also fixes the intermediate vertex associated with
\[
\psi(x_1, x_2, x_3) = \inf\left\{\nu(x_1) - \frac{1}{4}, \nu(x_2), \nu(x_3) + \frac{1}{4}\right\}.
\]
Then one checks by a direct matrix computation that $g_0 \Gam_{\Sig\mu}(\frakp_5^2) g_0^{-1} \cap \Gam_{\Sig\mu}(\frakp_5^2)$ contains $\Gam(\frakp_5^4)$. Indeed, note that $\frakp_5^4 = 5 \calO_E$ and if
\[
\gam = \left(\begin{smallmatrix}
1 + 5 c_3 & -(\zeta_5^2 + 2 \zeta_5 + 1) \pi^5 c_2 & -(\zeta_5^3 + 2 \zeta_5^2 + \zeta_5) \pi^6 c_1 \\
(\zeta_5^3 + \zeta_5^2 - 1) \pi^3 b_3 & 1 + 5 b_2 & -(\zeta_5^3 + 2 \zeta_5^2 + 2 \zeta_5 + 1) \pi^5 b_1 \\
(\zeta_5^3 + 2 \zeta_5^2 + 2 \zeta_5 + 1) \pi^2 a_3 & (\zeta_5^3 + 2 \zeta_5^2 + \zeta_5) \pi^3 a_2 & 1 + 5 a_1
\end{smallmatrix}\right),
\]
then $\gam \in \Gam_{\Sig \mu}(\frakp_5^2)$ and
\[
g_0 \gam g_0^{-1} = \begin{pmatrix}
1 + 5 a_1 & 5 a_2 & 5 a_3 \\
5 b_1 & 1 + 5 b_2 & 5 b_3 \\
5 c_1 & 5 c_2 & 1 + 5 c_3
\end{pmatrix} \in \Gam_{\Sig \mu}(\frakp_5^4).
\]
From here, one applies the techniques developed in the previous examples to prove Theorem \ref{thm:Kahler}.

\begin{rem}
Analogous to our realization of the building for $\PO(q; \bbQ_2)$ inside the building for $\PGL_5(\bbQ_2)$ in \S\ref{ssec:moreBianchi}, we can realize the building for $\PU(h, E_5)$ inside the building for $\PGL_3(E_5)$ by taking the vertex associated with an additive norm $\phi$ to be the homothety class of the $\calO_E$ lattice on which $\phi$ takes nonnegative values. In the above notation and recalling that $\nu$ is normalized to have value group $\frac{1}{2} \bbZ$, this gives:
\begin{align*}
\phi_0 &\mapsto [\langle e_1, e_2, e_3 \rangle] \\
\phi_1 &\mapsto [\langle \pi e_1, e_2, \pi^{-1} e_3 \rangle] \\
\psi &\mapsto [\langle \pi e_1, e_2, e_3 \rangle]
\end{align*}
One can proceed as in \S\ref{ssec:moreBianchi} to compute the fundamental apartment associated with the standard $\bbQ_5$-split torus and compute the valence of each vertex of the tree.
\end{rem}

\medskip

Combining Theorem \ref{thm:Kahler} and \cite[Thm.\ E]{FV}, one obtains a new proof of the following (which was known by \cite[Thm.\ 3]{StoverBetti} without knowing which congruence tower contains the fibration).

\begin{cor}\label{cor:CHfiber}
The group $\Gam_{\Sig\mu}$ virtually algebraically fibers on a congruence subgroup of level dividing $5$.
\end{cor}

\begin{rem}\label{rem:NotNonuniform}
We note that nonuniform lattices in $\PU(n,1)$ cannot be virtually RFRS for $n \ge 2$. This is because their cusp subgroups are virtually two-step nilpotent groups, but two-step nilpotent groups are not virtually RFRS and being RFRS descends to subgroups. However, if $\Gam < \PU(n,1)$ was a nonuniform arithmetic lattice contained in the congruence subgroup of level $\frakp$ for which $\Gam^{ab}$ contains no $p$-torsion, where $\frakp$ is a prime of residue characteristic $p$, then the methods of this paper would produce a congruence RFRS tower, which is impossible. In particular, we conclude that $\Gam^{ab}$ must have $p$-torsion.

One way to find this $p$-torsion is as follows. Since $\Gam$ is contained in a congruence subgroup, away from some small exceptions the associated complex hyperbolic manifold $\bbB^n / \Gam$ admits a \emph{smooth toroidal compactification} in the sense of \cite{AMRT}. One often sees that the cusp cross-sections of $\bbB^n / \Gam$ are nil-manifolds with $p$-torsion in their homology. For example, for $\Gam(\frakp)$ the center of any peripheral subgroup generates $p$-torsion in the homology of the associated nil-manifold. Careful consideration of the standard Mayer--Vietoris sequence for the toroidal compactification (cf.\ \cite[\S4]{DiCerboStover}) allows one to then conclude that this $p$-torsion in the homology of the cusp cross-section must in fact induce $p$-torsion in the homology of $\bbB^n / \Gam$.

In particular, peripheral subgroups of $\Gam$ can force $\Gam^{ab}$ to have $p$-torsion when $\Gam$ is contained in the congruence subgroup of level $\frakp$. Thus the obstruction to $\Gam$ containing a RFRS tower is also an obstruction to $\Gam^{ab}$ having no $p$-torsion.
\end{rem}

\section{Conclusion}\label{sec:Conclusion}

There are many natural questions that arise from the results and methods of this paper.

\medskip

We recall that a group $\Gam$ is said to \emph{virtually algebraically fiber} if it has a homomorphism onto $\bbZ$ with finitely generated kernel. This is an algebraic generalization of the well-known Stallings criterion for a compact $3$-manifold to fiber over $S^1$.

\begin{qtn}
Which commensurability classes of rank $1$ arithmetic lattices contain a congruence subgroup that is algebraically fibered?
\end{qtn}

This question was originally posed by Baker and Reid in personal communication.  A $4$-dimensional lattice that virtually algebraically fibers was given in \cite[Rem.\ 5.3]{JNW}, though we do not know if the example fibers on a congruence subgroup. An obvious obstruction to having a virtual algebraic fibration on a congruence subgroup is if every congruence lattice in the commensurability class has trivial $1^{st}$ betti number. For example, Bergeron and Clozel proved that the first betti number vanishes for all congruence arithmetic lattices in $\PO(7,1)$ defined via triality \cite[Thm.\ 1.1]{BergeronClozel3}. For all other arithmetic lattices in $\PO(n,1)$, $n\neq 3$, one can find a congruence subgroup with nontrivial $1^{st}$ betti number \cite[Cor.\ 1.8]{BergeronClozel2} (the $n=3$ case is open - see \cite{Schwermer} for a discussion of what is known). There are also classes of arithmetic lattices in $\PU(n,1)$ where each kind of behavior occurs. See \cite{BergeronClozel1} for more on what is known for cohomological vanishing for congruence arithmetic lattices in $\PU(n,1)$ and \cite{StoverBetti} for more on algebraic fibrations in that setting.

\begin{qtn}
Given a congruence arithmetic group, how often does a principal congruence subgroup at a prime ideal $\mathfrak{p}$ have no $p$-torsion in $H_1$, where $\mathfrak{p}|p$? Is there some arithmetic significance to this phenomenon?
\end{qtn}

See Appendix A by \c{S}eng\"un for data indicating that vanishing of $p$-torsion is quite frequent for congruence subgroups of Bianchi groups, but by no means ubiquitous.

\begin{qtn}
For each $n > 1$, is there a prime $p$ so that the congruence subgroup $\Gamma(p)$ of level $p$ in $\SO(n,1;\mathbb{Z})$ has no $p$-torsion in its abelianization?
\end{qtn}

If true, this would give a positive answer to Question 1 for arithmetic hyperbolic groups of simplest type (i.e., those defined by a quadratic form), since one can embed a congruence subgroup of these groups into $\SO(n,1;\mathbb{Z})$ by restriction of scalars.

\begin{qtn}
When does $\Gamma(\mathfrak{p}^n)$, $n\in \mathbb{N}$, form a RFRS sequence? Is the sequence RFRS whenever it is at the first stage, i.e., $\Gamma(\mathfrak{p})^{(1)}_r \le \Gamma(\mathfrak{p}^2)$?
\end{qtn}

This is  roughly a version of another question posed by Baker and Reid in private communication. In this paper, we only show that an interlacing of this sequence is RFRS if $H_1(\Gamma(\mathfrak{p}); \bbZ)$ has no $p$-torsion. Also recall Remark \ref{rem:NotNonuniform}.

\begin{qtn}
When does this strategy work for nonarithmetic hyperbolic lattices? When is there a congruence subgroup that fibers, or a congruence RFRS tower?
\end{qtn}

Note that any lattice in $\PO(n,1)$, $n \ge 3$, or $\PU(n,1)$, $n \ge 2$, is a subgroup of an $S$-arithmetic group by local rigidity. Indeed, the lattice can be embedded in $\GL_n(K)$ for $K$ a number field, and hence lies in $\GL_n(\calO)$ for $\calO$ some finitely generated subring of $K$. Therefore, the notion of congruence subgroup makes sense when one avoids the primes in $S$, where $S$ denotes the primes that are inverted in $\calO$.

\pagebreak

\section*{Appendix A\\ Torsion in the homology of principal congruence subgroups of Bianchi groups}

\noindent
{\large
Mehmet Haluk \c{S}eng\"un}\\
University of Sheffield\\
\texttt{m.sengun@sheffield.ac.uk}

\renewcommand{\thesection}{\Alph{section}}
\setcounter {section}{1}

\subsection{Introduction}
Let $K$ be an imaginary quadratic field with ring of integers $\bbZ_K$. An ideal $\fraka$ of $\bbZ_K$ determines a finite-index normal subgroup $\Gamma(J)$ of the Bianchi group $\SL_2(\bbZ_K)$, called the {\em principal congruence subgroup of level $\fraka$}. If $\frakp$ is a prime idea of $\bbZ_K$ over the rational prime $p$, the question of whether the abelian group $H_1(\Gamma(\frakp),\bbZ)$ has $p$-torsion arises naturally in the current work of Ian Agol and Matthew Stover. In this appendix, we try to gain insight into this question by producing numerical data.

\subsection{Methodology}
Let $K$ be one of the five imaginary quadratic fields for which $\bbZ_K$ is Euclidean, namely $K=\bbQ(\sqrt{-d})$ with $d=1,2,3,7,11$. Let $\frakp$ be a prime ideal of $\bbZ_K$. Our starting point is the basic fact that $H_1(\Gamma(\frakp),\bbZ) \simeq \Gamma(\frakp)^{ab}$ where $\Gamma(\frakp)^{ab}$ is the abelianization of $\Gamma(\frakp)$. To compute the abelianization of $\Gamma(\frakp)$, we will need a presentation. We will obtain this presentation from a presentation of $\SL_2(\bbZ_K)$ using the standard functions in the Finitely Presented Groups package of the computer algebra system {\sf Magma}.

Presentations for Bianchi groups go back to the late 19th century. We prefer to use those given in \cite[p.37]{grunewald_etal}. The presentations given there are for the projective Bianchi groups $\PSL_2(\bbZ_K)$. To obtain a presentation for $\SL_2(\bbZ_K)$, we simply introduce another generator $j= \left ( \begin{smallmatrix} -1 & 0 \\ 0 & -1 \end{smallmatrix} \right )$, modify the existing relations accordingly and add new relations to ensure that $j$ is central. We present here the result for the case $K=\bbQ(\sqrt{-1})$:
\begin{align*} 
\SL_2(\bbZ_K) = \langle a,b,u,j  \mid & (ab)^3=j, b^2=j, j^2=1, [a,u]=1, (bubu^{-1})^3=1, \\ 
& j=(bu^2bu^{-1})^2, j=(aubau^{-1}b)^2,  [a,j]=1, [u,j]=1 \rangle .
\end{align*} 
We have the matrix realizations $a=\left ( \begin{smallmatrix} 1 & 1 \\ 0 & 1 \end{smallmatrix} \right )$, $b=\left ( \begin{smallmatrix} 0 & -1 \\ 1 & 0 \end{smallmatrix} \right )$ and $u=\left ( \begin{smallmatrix} 1 & \sqrt{-1} \\ 0 & 1 \end{smallmatrix} \right )$.  

The principal congruence subgroup $\Gamma(\frakp)$ is the kernel of the surjective homomorphism 
\[ \SL_2(\bbZ_K) \longrightarrow  \SL_2(\bbZ_K/\frakp), \quad \left ( \begin{smallmatrix} a & b \\ c & d \end{smallmatrix} \right ) \mapsto \left ( \begin{smallmatrix} \bar{a} & \bar{b} \\ \bar{c} & \bar{d} \end{smallmatrix} \right )  
\]
where $x \to \bar{x}$ is the reduction map $\bbZ_K \to \bbZ_K/\frakp$. We implement this homomorphism in {\sf Magma} and ask {\sf Magma} to compute its kernel. Given the presentation of $\SL_2(\bbZ_K)$, {\sf Magma} then can compute a presentation for $\Gamma(\frakp)$ using Reidemeister--Schreier type algorithms. Finally we ask {\sf Magma} to compute the abelianization. For the readers' convenience, we make our code public on our homepage.\footnote{https://sites.google.com/site/mhaluksengun/research}

\subsection{Results}
As mentioned above, we compute with prime ideals. As $H_1(\Gamma(\frakp), \bbZ) \simeq H_1(\Gamma(\overline{\frakp}),\bbZ)$, for prime ideals with $\frakp \not= \overline{\frakp}$ (here), we computed with only one of them.  We list the norm of the prime ideal $\frakp$, the rank of $H_1(\Gamma(\frakp), \bbZ)$ and the size of the torsion subgroup of $H_1(\Gamma(\frakp), \bbZ)$. The size is given in its prime factorisation. 

{\small
\[
\begin{tabular}{|r|r|c|}
\hline
Norm($\frakp$) & rank & size of torsion \\ \hline
\multicolumn{3}{c}{$K=\bbQ(\sqrt{-1})$} \\ \hline
2 &   0 &  $2^5$\\
5 &   6 &   $1$\\
9 &  20 &    $1$\\
13 &  42 &   $1$\\
17 &  72 &    $1$\\
29 & 238 & $3^1$\\
37 & 342 &    $1$\\
41 & 420 & $2^{62}$\\
49 & 825 & $7^6$\\
53 & 702 & $3^{104}$\\
61 & 930 & $2^{124} 29^{62}$\\
73 & 1332 & $3^{37} 5^{74} 19^{74}$\\ \hline
\multicolumn{3}{c}{$K=\bbQ(\sqrt{-2})$} \\ \hline
2&  3&$2^2$ \\
3&  4&   $1$ \\
11& 60&   $1$ \\
17&144& $2^9$ \\
19&180&   $1$ \\
25&403&$5^7$ \\
41&881&$2^{85} 5^1 7^{40} 127^{40}$ \\
43&924&$2^{88} 3^{42} 67^{42} 127^{44}$ \\
49&1724&$7^{133}$ \\
59&1740&$2^{290} 3^{58} 11^{116} 31^{58} 59^{236} 5743^{60}$ \\
67&2244&$2^{200} 3^{135} 239^{66} 271^{66} 647^{66} 727^{68} 38011^{66} 47917^{68}$ \\ 
73 & 2738 & $2^{369} 3^{296} 19^{73} 73^{11} 1511^{74} 2089^{74} 22051^{72} 150959^{72}$ \\ \hline 
\multicolumn{3}{c}{$K=\bbQ(\sqrt{-3})$} \\ \hline
3&  0& $3^3$ \\ 
4&  5& $2^1$ \\ 
7&  8&  $1$ \\
13& 28&  $1$ \\
19& 60&  $1$ \\
25&117&  $1$ \\
31&160&  $1$ \\
37&228& $3^{19}$ \\ 
43&308&$2^{44}$ \\ 
61&620&$3^{62}$ \\ 
67&748&$2^{68} 5^{67}$ \\ 
79&1040&$2^{80} 41^{78}$ \\ \hline
\end{tabular}
\]
}

{\scriptsize
\[
\begin{tabular}{|r|r|c|}
\hline
Norm($\frakp$) & rank & size of torsion \\ \hline
\multicolumn{3}{c}{$K=\bbQ(\sqrt{-7})$} \\ \hline
2 &  3 & $2^1$ \\ 
7& 24&  $1$ \\ 
9 & 40 & $3^1$ \\
11& 60&   $1$ \\ 
23&264& $2^{22}$ \\ 
25 & 376 & $5^7$ \\
29&420& $5^{15}$ \\ 
37&684& $3^{19}  19^{38}$ \\ 
43&924& $5^{87} 41^{44} 67^{42}$ \\   
53&1404& $2^{320}  37^{52} 593^{52}  857^{54}$ \\  
67&2244& $2^{1056} 3^{67} 11^{67}  89^{68} 131^{68}  137^{66} 463^{68}$ \\  
71&2520& $2^{288} 5^{140} 29^{72} 59^{70}  89^{70}  311^{70} 937^{70} 19319^{72}$ \\ \hline
\multicolumn{3}{c}{$K=\bbQ(\sqrt{-11})$} \\ \hline
3&  4&  $1$ \\ 
4 &  15 & $2^2$ \\ 
5& 12&  $1$ \\ 
11& 81&  $1$ \\ 
23&264&$ 2^{112}$ \\ 
31&480&$2^{96} 5^{30} 29^{32} 31^{30}$ \\ 
37&722&$17^{38} 19^{36} 37^{5} 683^{36} $ \\ 
47&1151&$3^{138} 5^{47} 17^{46} 23^{1} 37^{46} 97^{46} 191^{46} 1609^{48} $ \\ 
53&1404&$2^{320} 5^{160} 11^{27} 19^{104} 53^{52} 431^{52} 683^{52} 859^{54} $ \\ 
59&1740&$2^{348} 3^{408} 5^{58} 7^{116} 11^{58} 17^{120} 19^{58} 31^{58} 59^{58} 199^{58} 233^{60} 5279^{60} 20341^{58} $ \\ 
67&2244&$2^{136} 3^{132} 17^{133} 31^{68} 67^{66} 197^{68} 331^{68} 613^{66} 2309^{68} 5807^{68} 67829^{66} 256189^{66} $ \\ \hline

\end{tabular}
\]
}
\bibliography{CongruenceRFRS}

\end{document}